\documentclass{article}
\usepackage[utf8]{inputenc}
\usepackage[T1]{fontenc}
\usepackage{authblk}
\usepackage{comment}
\usepackage[utf8]{inputenc}
\usepackage{authblk}
\usepackage[T1]{fontenc}
\usepackage{amssymb}
\usepackage{amsmath, amsthm}
\usepackage{amsfonts}
\usepackage{graphics}
\usepackage{color}
\usepackage{xcolor}
\usepackage{graphicx}
\usepackage{graphics}
\usepackage{color}
\usepackage{amsmath}
\usepackage{amsfonts}
\usepackage{amssymb}
\usepackage{enumerate}
\usepackage{hyperref}
\newtheorem{Statement}{Statement}

\newtheorem{Proposition}[Statement]{Proposition}

\newtheorem{Theorem}[Statement]{Theorem}
\newtheorem{Conjecture}[Statement]{Conjecture}

\newtheorem{Claim}[Statement]{Claim}
\newtheorem{Lemma}[Statement]{Lemma}

\title{Rainbow bases in matroids}
\author{Florian H\"orsch\footnote{ Most of the work was done while this author was part of TU Ilmenau, Germany.}\\ CISPA, Saarbrücken, Germany \\ Florian.Hoersch@cispa.de\\
\bigskip 

Tom\'a\v{s} Kaiser  \footnote{ Supported
by project GA20-09525S of the Czech Science Foundation.}\\ University of West Bohemia, Plze\v{n}, Czech Republic \\ kaisert@kma.zcu.cz\\ 

\bigskip

Matthias Kriesell \\ TU Ilmenau,  Germany \\ matthias.kriesell@tu-ilmenau.de}
\begin{document}

\maketitle

\begin{abstract}
Recently, it was proved by B\'erczi and Schwarcz that  the problem of factorizing a matroid into rainbow bases with respect to a given partition of its ground set is algorithmically intractable. On the other hand, many special cases were left open.

We first show that the problem remains hard if the matroid is graphic, answering a question of  B\'erczi and Schwarcz. As another special case, we consider the problem of deciding whether a given digraph can be factorized into subgraphs which are spanning trees in the underlying sense and respect upper bounds on the indegree of every vertex. We prove that this problem is also hard. This answers a question of Frank.

In the second part of the article, we deal with the relaxed problem of covering the ground set of a matroid by rainbow bases. Among other results, we show that there is a linear function $f$ such that every matroid that can be factorized into $k$ bases for some $k \geq 3$ can be covered by $f(k)$ rainbow bases if every partition class contains at most 2 elements.
\end{abstract}%_______________________________________________________________
\section{Introduction}
This article deals with some problems on factorizing the common ground set of two matroids. Any undefined notions can be found in Section \ref{preldef}.

Several problems involving common bases of two matroids have played a significant role in the history of combinatorial optimization. Nevertheless, the algorithmic complexity of the question whether the common ground set of two given matroids can be factorized into a collection of $k$ common bases remained open for a long time. An important result of Harvey, Kir\'aly and Lau \cite{hkl} shows the equivalence of this problem to a seemingly much less general one. Namely, they show that this general problem is equivalent to its special case when one of the two matroids is a {\it unitary partition matroid}, that  is, the direct sum of uniform matroids of rank 1. Such a matroid can be thought of as a partition of the elements of the first matroid and a common basis of both matroids is a basis of the first matroid which is {\it rainbow} with respect to that partition, meaning it contains at most one element of each partition class.

 The above problem hence reduces to the question whether given a matroid $M$ and a partition of $E(M)$, we can factorize $M$ into $k$ rainbow bases.

Unfortunately, this latter problem has been proven to be algorithmically intractable by B\'erczi and Schwarcz \cite{bs}. Firstly, they showed that this problem cannot be solved efficiently in the independence oracle model and secondly, they showed that some NP-hard problems appear as special cases of this problem.

Despite the results of B\'erczi and Schwarcz showing the intractability of the general matroid factorization problem, there are still many natural combinatorial problems appearing as special cases of the matroid factorization problem whose tractability is interesting in its own right. One of the most important examples is the problem of factorizing the arc set of a given digraph $D$ into a collection of arc-disjoint $r$-arborescences for some given root vertex $r \in V(D)$. This problem has been dealt with by Edmonds \cite{e4}. It can be viewed as the special case of the problem of factorizing a given matroid into rainbow bases with respect to a given partition of its elements, where the  matroid is the cycle matroid of a graph $G$ and every partition class of the partition is associated to a vertex $v \in V(G)$ in such a way that every vertex of $G$ except one is associated to exactly one partition class and the edges of each partition class induce a star at their associated vertex. The crucial observation of Edmonds is that in this setting, a factorization of the matroid into $k$ rainbow bases exists if and only if two natural necessary conditions are satisfied. Firstly, a factorization of the matroid into $k$ bases must exist when disregarding the partition of its element set. Such a matroid is called a {\it $k$-base matroid} and in the graphic case, we speak of a {\it $k$-multiple tree}. Note that 2-base matroids are also referred to as block matroids. Secondly, every partition class needs to be of size at most $k$.

These two conditions do not continue to guarantee a factorization into rainbow bases when the conditions on the particular structure of the matroid and the classes of the partition are dropped. As mentioned by Chow \cite{chow} and by Fekete and Szab\'o (\cite{fs},\cite{fsj}), for a first well-known counterexample, we can consider the cycle matroid of $K_4$ together with the partition in which every pair of non-adjacent edges forms a partition class. More generally, it is not hard to see that  counterexamples can be obtained by considering the cycle matroid of an odd wheel together with the unique partition of its  element set in which each partition class contains the elements corresponding to one spoke edge and the rim edge that is antipodal to the endpoint of the spoke edge on the outer cycle. For an illustration, see Figure \ref{drftg}.

\begin{figure}[h]
    \centering
        \includegraphics[width=.8\textwidth]{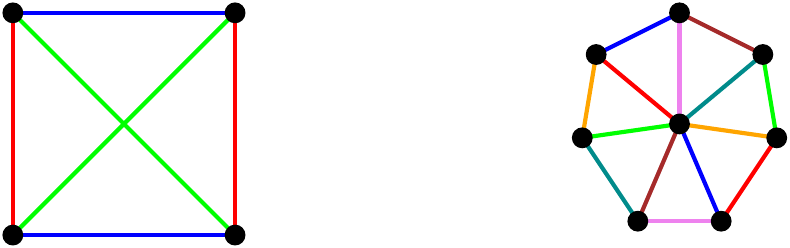}
        \caption{A $K_4$ and a wheel on 8 vertices together with the described edge partition.}\label{drftg}
\end{figure}

Brualdi and Hollingsworth \cite{bh} conjectured that $K_4$ is the only negative instance among a class of similar graphic matroids with edge partitions.  The conjecture states that if $M$ is the cycle matroid of the complete graph $K_n$ for some even integer $n \geq 6$ and every partition class corresponds to a perfect matching of this graph, then a factorization of $M$ into rainbow bases exists. Several partial results on this conjecture have been obtained, see \cite{h}, \cite{ps} and \cite{gkmo}.

In the light of the results of B\'erczi and Schwarcz and of Edmonds, it is natural to ask about the complexity of factorizing a given matroid into rainbow bases with respect to a given partition when the conditions are somewhat relaxed in comparison to Edmonds' setting. The first contribution of this article is to provide two negative results in that direction. \footnote{After this
paper was submitted for publication, we were informed by K. Bérczi that
related results were independently obtained in the paper \cite{bck} by Bérczi,  Cs\'aji, and Kir\'aly. This
includes the $k=2$ case of Theorem 1.}

In the first setting, we relax the condition of every partition class inducing a star at a vertex and only maintain the condition of the matroid being graphic. In a graph-theoretical formulation, we consider the problem RST$k$F in which we are given a $k$-multiple tree $G$ and a partition of $E(G)$ and need to decide whether $G$ can be factorized into a collection of $k$ rainbow spanning trees. Observe that this setting includes the instances depicted in Figure \ref{drftg}. Using a reduction that builds on the example of $K_4$ depicted in Figure \ref{drftg}, we provide the following result.

\begin{Theorem}\label{hard1}
RST$k$F is NP-hard for any $k \geq 2$.
\end{Theorem}
This answers a question of B\'erczi and Schwarcz \cite{bs}.
\medskip

Next, we deal with a setting in which, in comparison to Edmonds' setting, the condition of the partition matroid being unitary is relaxed. In a graph-theoretical formulation, we are given a digraph $D$ that is an orientation of a $k$-multiple tree and such that for every $v \in V(D)$, there is an integer $\alpha_v$ such that $d_D^-(v)=\alpha_v k$. We then need to decide whether $D$ can be factorized into a collection of subgraphs $(H_1,\ldots, H_k)$ such that each $H_i$ is a spanning tree in the underlying sense and satisfies $d_{H_i}^-(v)=\alpha_v$ for all $v \in V(D)$. We denote this problem by BST$k$F. Frank \cite{f1} conjectured that this problem could be dealt with in a similar way as the one in Edmonds' setting. This conjecture was refuted by Kir\'aly \cite{t} by pointing to a result of Fekete and Szab\'o \cite{fs}. Later, Frank \cite{book} asked about the computational complexity of BST$k$F. We give the following answer to this question using Theorem \ref{hard1} for a reduction.

 \begin{Theorem}\label{hard2}
BST$k$F is NP-hard for any $k \geq 2$.
\end{Theorem}
%This answers a question of Frank \cite{book}.

One interesting feature of the examples in Figure \ref{drftg} and also of the negative instances created in the reductions proving Theorems \ref{hard1} and \ref{hard2} is that they always contain some partition classes whose size equals the number of bases we wish to factorize the matroid into. We would like to understand whether the situation gets brighter when the size of the partition classes is small in comparison to the number of bases we wish to factorize the matroid into. Observe that this problem no longer corresponds to factorizing the common ground set of an arbitrary matroid and a unitary partition matroid into common bases. On the other hand, it corresponds to both the factorization of the ground set of an arbitrary matroid into bases which are independent in a unitary partition matroid and to the factorization of the ground set into common bases of the arbitrary matroid and a truncation of the unitary partition matroid. We here give the following, rather optimistic conjecture:

\begin{Conjecture}\label{optim}
Let $M$ be a $k$-base matroid for some positive integer $k$ and let $\mathcal{P}$ be a partition of $E(M)$ such that $|X|\leq k-1$ holds for every $X \in \mathcal{P}$. Then $M$ can be factorized into $k$ rainbow bases.
\end{Conjecture}

Observe that the case $k=2$ is trivial. We hence focus on the first non-trivial case, namely the partition classes being of size 2. We believe that an answer to the following special case of Conjecture \ref{optim} would already be very interesting.

\begin{Conjecture}\label{optim2}
Let $M$ be a $k$-base matroid for some $k \geq 3$ and let $\mathcal{P}$ be a partition of $E(M)$ such that $|X|\leq 2$ holds for every $X \in \mathcal{P}$. Then $M$ can be factorized into $k$ rainbow bases.
\end{Conjecture}

K. Bérczi and T. Király kindly pointed out to us that Conjectures \ref{optim} and \ref{optim2}
are implied by the following conjecture of Aharoni and Berger, stated as
Conjecture 1.4 in \cite{ABZ} and attributed there to \cite{AB} (see
also \cite[Conjecture~1.5]{AH} for a slightly different version):
\begin{Conjecture}\label{ahconj}
If $M$ and $N$ are matroids on the same ground set $E$ such that $E$ can
be covered by $k$ independent sets of $M$ as well as by $\ell$
independent sets of $N$ where $1\leq k\leq\ell$, then $E$ can be
covered by $\max\{k+1,\ell\}$ common independent sets of $M$ and $N$.
\end{Conjecture}

The second contribution of this article is a collection of results making some progress toward Conjecture \ref{optim2}. Observe that factorization problems are the intersection of two different classes of problems, namely packing and covering problems. When we are given a $k$-base matroid and a partition of its ground set, the questions whether this matroid can be factorized into $k$ rainbow bases, whether it contains a packing of $k$ rainbow bases and whether it can be covered by $k$ rainbow bases are all equivalent. Hence Theorems \ref{hard1} and \ref{hard2} imply hardness results for the packing and covering versions of these problems. Interestingly, most attempts to approach problems of factorization of matroids  into common bases have so far used the packing perspective. For example, most of the aforementioned partial results for the conjecture of Brualdi and Hollingsworth are packing results. Horn and Nelsen \cite{hn} approach a problem which is closely related to Conjectures \ref{optim} and \ref{optim2} through a packing result relying on algebraic methods.

We take a different approach, namely we give some approximate versions for Conjecture \ref{optim2} from the covering perspective. We give the following result for $k \geq 4$. We say that a partition is {\it $p$-bounded} if each of its partition classes contains at most $p$ elements.

\begin{Theorem}\label{k41}
Let $M$ be a $k$-base matroid where $k\geq 4$ and $k=2 \alpha+ \beta$ for some nonnegative integers $\alpha$ and $\beta$ and $\mathcal{P}$ a 2-bounded partition of $E(M)$. Then $E(M)$ can be covered by $5 \alpha+4\beta$ bases of $M$ which are rainbow  with respect to $\mathcal{P}$.
\end{Theorem}

In order to extend our proof for $k=3$, we need to slighly alter our proof technique. This leads to the following somewhat worse bound.

\begin{Theorem}\label{k412}
Let $M$ be a $3$-base matroid and $\mathcal{P}$ a 2-bounded partition of $E(M)$. Then $E(M)$ can be covered by 13 bases of $M$ which are rainbow  with respect to $\mathcal{P}$.
\end{Theorem}

%Let $M$ be a $3$-base matroid and $\mathcal{P}$ a 2-bounded partition of $E(M)$. Then $E(M)$ can be covered by $16$ bases of $M$ which are rainbow  with respect to $\mathcal{P}$.
%\end{Theorem}

%\begin{Theorem}\label{DEd}
%Let $M$ be a graphic $3$-base matroid and $\mathcal{P}$ a 2-bounded partition of $E(M)$. Then $E(M)$ can be covered by $12$ bases of $M$ which are rainbow  with respect to $\mathcal{P}$.
%\end{Theorem}

In the light of these results, the question whether they could be extended to $k=2$ becomes one of independent interest. 
Unfortunately, our techniques do not yield a similar result for this case. However, as a first step toward such a theorem, we are able to show the following result.

\begin{Theorem}\label{doublecover}
Let $M$ be a $2$-base matroid and $\mathcal{P}$ a 2-bounded partition of $E(M)$. Then $E(M)$ can be covered by $O(\log(|E(M)|))$ bases of $M$ which are rainbow  with respect to $\mathcal{P}$.
\end{Theorem}

The rest of this article is structured as follows. In Section \ref{preldef}, we give some more formal definitions of the objects in consideration and list some preliminary results. In Section \ref{complex}, we give the proofs of Theorems \ref{hard1} and \ref{hard2}. Next, in Section \ref{cover}, we prove Theorems \ref{k41} to \ref{doublecover}. Finally, in Section \ref{conc}, we conclude with a discussion of our results.

\section{Definitions and preliminaries}\label{preldef}
In this section, we give all the notation and preliminary results we need for our main proofs. After giving some basic notation in Section \ref{nota}, we give formal descriptions of all the algorithmic problems considered in Section \ref{algo}. In Section \ref{bas}, we give some simple properties of graphs. In Section \ref{weed}, we describe a new class of graphs we need for one of our reductions in Section \ref{complex}.
\subsection{Notation}\label{nota}
When we speak of a collection of {\it disjoint} objects, we mean that they are pairwise disjoint.
%A mixed graph $H$ consists of a {\it vertex set} $V(H)$, an {\it edge set} $E(H)$ and an {\it arc set} $A(H)$. Every $e \in E(H)$ is a set containing two distinct elements of $V(H)$ called its {\it endvertices}, while every $a \in A(H)$ is a tuple containing two distinct elements of $V(H)$ called its {\it tail} and {\it head}, respectively. 
Graphs are considered to be undirected and graphs and digraphs may have parallel edges and arcs, respectively. Let $G$ be a graph. We denote the vertex set of $G$ by $V(G)$ and the edge set of $G$ by $E(G)$. For some disjoint $X,Y \subseteq V(G)$, we denote by $\delta_G(X,Y)$ the set of edges in $E(G)$ which have one endvertex in $X$ and one endvertex in $Y$. We abbreviate $\delta_G(X,V(G)-X)$ to $\delta_G(X)$. Further, we use $d_G(X,Y)$ for $|\delta_G(X,Y)|$ and $d_G(X)$ for $|\delta_G(X)|$. 

Let $D$ be a digraph. We denote the vertex set of $D$ by $V(D)$ and the arc set of $D$ by $A(D)$. For some disjoint $X,Y \subseteq V(D)$, we denote by $\delta_D(X,Y)$ the set of arcs in $A(D)$ whose tail is in $X$ and whose head is in $Y$ and we abbreviate $\delta_D(V(D)-X,X)$ to $\delta_D^-(X)$ . Further, we use $d_D(X,Y)$ for $|\delta_D(X,Y)|$ and $d_D^-(X)$ for $|\delta_D^-(X)|$. For a single vertex $v$, we often use $v$ instead of $\{v\}$.% If $E(H)=\emptyset$, then $H$ is called a {\it digraph} and if $A(H)=\emptyset$, then $H$ is called a {\it graph}.

 The graph that is obtained from a digraph $D$ by replacing every $a \in A(D)$ by an edge containing the same two vertices as $a$ is called the {\it underlying graph} $UG(D)$ of $D$. Given a graph $G$, an {\it orientation} of some $e \in E(H)$ is an arc containing the same two vertices as $e$. A digraph that is obtained from $G$ by replacing every $e \in E(G)$ by an orientation of itself is called an {\it orientation} of $G$.

Given a graph $H$, a {\it subgraph} of $H$ is a graph $H'$ with $V(H')\subseteq V(H)$ and $E(H')\subseteq E(H)$. Further, we say that $H'$ is {\it spanning} if $V(H')=V(H)$. A {\it factorization} of $H$ is a collection of spanning  subgraphs $(U_1,\ldots,U_k)$ of $H$ such that $\bigcup_{i=1}^kE(U_i)=E(H)$ and for all $i,j \in \{1,\ldots,k\}$ with $i \neq j$, we have $E(U_i)\cap E(U_j)= \emptyset$. We use similar definitions for digraphs.

A graph $T$ is called a {\it tree} if $d_T(X)\geq 1$ for all nonempty $X\subsetneq V(T)$ and $T$ is minimal with that property with respect to edge-deletion. Given a graph $G$, a spanning subgraph $T$ of $G$ is called a {\it spanning tree} of $G$ if $T$ is a tree. A graph that admits a factorization into $k$ spanning trees is called a {\it $k$-multiple tree}. We abbreviate a 2-multiple tree to a {\it double tree}. An orientation $\vec{T}$ of a tree $T$ in which $d_{\vec{T}}-v$ holds for all $v \in V(T)-r$ for some $r \in V(T)$ is called an {\it $r$-arborescence}.

For basics on matroid theory, see \cite{book}, Chapter 5. Given a matroid $M$, we use $r_M$ to denote the rank function of $M$ and refer to the elements of $M$ by $E(M)$. If there is a collection $(B_1,\ldots,B_k)$ of bases of $M$ such that $\bigcup_{i=1}^k B_i=E(M)$ and $B_i\cap B_j= \emptyset$ for all $i,j \in\{1,\ldots,k\}$ with $i \neq j$, we say that $M$ is a {\it $k$-base matroid} and $(B_1,\ldots,B_k)$ is a {\it factorization} of $M$ into bases. A {\it partition matroid} is a matroid that is the direct sum of uniform matroids. Moreover, if each of these uniform matroids is of rank 1, it is a {\it unitary partition matroid}. If a matroid is isomorphic to the cycle matroid of a graph, we call it {\it graphic}.

Given a partition $\mathcal{P}$ of a ground set $E$ and an integer $p$, we say that $\mathcal{P}$ is {\it $p$-bounded} if $|X|\leq p$ for all $X \in \mathcal{P}$, and that $\mathcal{P}$ is {\it $p$-uniform} if $|X|= p$ for all $X \in \mathcal{P}$. Further, we say that a subset $S$ of $E$ is {\it rainbow} with respect to $\mathcal{P}$ if $|S \cap X|\leq 1$ for all $X \in \mathcal{P}$. We do not specify the partition if it is clear from the context.

\subsection{Algorithmic problems}\label{algo}
We here give the formal descriptions of the algorithmic problems considered in this article. The first two problems are those considered in Theorems \ref{hard1} and \ref{hard2}.

\medskip

\noindent \textbf{Rainbow Spanning Tree $k$-Factorization (RST$k$F):}
\smallskip

\noindent\textbf{Input:} A $k$-multiple tree $G$, a $k$-uniform partition $\mathcal{P}$ of $E(G)$.
\smallskip

\noindent\textbf{Question:} Can $G$ be factorized into $k$ rainbow spanning trees?
\medskip

\medskip

\noindent \textbf{Bounded Spanning Tree $k$-Factorization (BST$k$F):}
\smallskip

\noindent\textbf{Input:} A digraph $D$, a function $g:V(D) \rightarrow \mathbb{Z}_{\geq 0}$.
\smallskip

\noindent\textbf{Question:} Can $D$ be factorized into $k$ arc-disjoint subgraphs $(T_1,\ldots,T_k)$ such that for all $i \in \{1,\ldots,k\}$, $UG(T_i)$ is a spanning tree of $UG(D)$ and $d^-_{A(T_i)}(v)\leq g(v)$?
\medskip

The last two problems are useful for the reduction in the proof of Theorem \ref{hard1}.
\medskip

\noindent \textbf{Monotonic NAE3SAT (MNAE3SAT):}
\smallskip

\noindent\textbf{Input:} A set of variables $X$, a set of clauses $\mathcal{C}$ each containing 3 nonnegated variables of $X$.
\smallskip

\noindent\textbf{Question:} Is there a truth assignment $\phi:X \rightarrow \{true,false\}$ such that each clause contains at least one true and at least one false literal?
\bigskip

\noindent \textbf{$k$-Colorability ($k$COL):}
\smallskip

\noindent\textbf{Input:} A graph $G$.
\smallskip

\noindent\textbf{Question:} Is there a proper vertex-coloring of $G$ using at most $k$ colors?
\medskip

The following results can be found in \cite{schaefer} and \cite{gj}, respectively.

\begin{Proposition}\label{sathard}
MNAE3SAT is NP-hard.
\end{Proposition}

\begin{Proposition}\label{colhard}
$k$COL is NP-hard for any $k \geq 3$.
\end{Proposition}

\subsection{Basic graph properties}\label{bas}
We here give a collection of basic graph properties which will be useful for the reductions in Section \ref{complex}. Due to their simplicity, we omit their proofs.
\begin{Proposition}\label{ident}

Let $G_1,\ldots,G_t$ be a set of graphs whose vertex sets are disjoint and for every $i=1,\ldots,t$, let $T_i$ be a subgraph of $G_i$ and $v_i \in V(T_i)$. Further, let $G$ be obtained from the union of the graphs $G_i$ by identifying the vertices $v_i$ for $i=1,\ldots,t$ and let $T$ be obtained from the union of the graphs $T_i$ by identifying the vertices $v_i$ for $i=1,\ldots,t$. Then $T$ is a spanning tree of $G$ if and only if $T_i$ is  a spanning tree of $G_i$ for $i=1,\ldots,t$.
\end{Proposition}
\begin{Proposition}\label{fzf}
Let $G$ be a graph isomorphic to $K_4$ and $v \in V(G)$. Further, let $(F_1,F_2)$ be a partition of $\delta_G(v)$. Then there is a spanning tree factorization $(T_1,T_2)$ of $G$ such that $F_1 \subseteq E(T_1)$ and $F_2 \subseteq E(T_2)$ if and only if $F_1 \neq \emptyset$ and $F_2 \neq \emptyset$.
\end{Proposition}
The next result is an immediate consequence of the fact that the edge set of a complete regular bipartite graph can be partitioned into perfect matchings.
\begin{Proposition}\label{compbip}
Let $G$ be a complete bipartite graph with partition classes $A$ and $B$ where $|A|=k$ and $|B|=k-1$ for some $k \in \mathbb{Z}_+$. Further, let $\sigma:A \rightarrow \{1,\ldots,k\}$ be a bijection. Then there is a function $\phi:E(G)\rightarrow \{1,\ldots,k\}$ such that $\phi(\delta_{G}(b))=\{1,\ldots,k\}$ for all $b \in B$ and $\phi(\delta_{G}(a))=\{1,\ldots,k\}-\sigma(a)$ for all $a \in A$.
\end{Proposition}
\subsection{Marihuana leaves}\label{weed}

We now give a construction of a class of graphs that plays a significant role in the reduction in Section \ref{redtrees}. This construction is first mentioned by Chow in \cite{chow} where it is attributed to McDiarmid.

A {\it $k$-marihuana leaf} is obtained from $K_4$ by choosing a vertex $v$ and replacing each edge incident to $v$ by $k$ parallel copies of itself. Observe that $K_4$ is a 1-marihuana leaf. An illustration can be found in Figure \ref{fig0}.

\begin{figure}[h]
    \centering
        \includegraphics[width=.5\textwidth]{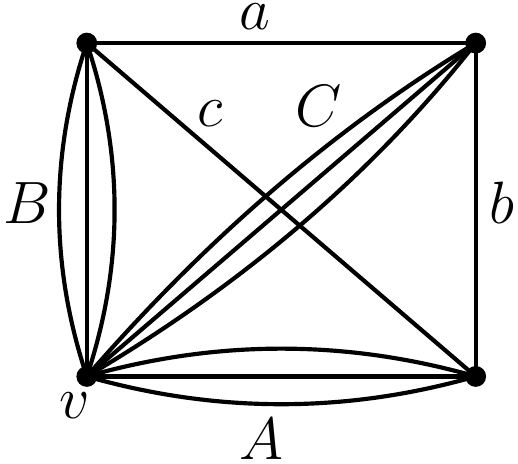}
        \caption{A 3-marihuana leaf whose edges are labeled as in Proposition \ref{long}.}\label{fig0}
\end{figure}

We now collect some properties of marihuana leaves.
\begin{Proposition}\label{long}
Let $G$ be a $(k-1)$-marihuana leaf. Further, let $A$ be a set of $k-1$ parallel edges and $a$ be the unique edge disjoint from the edges in $A$. We similarly define $B,b,C$, and $c$ in a way such that $A \cup B \cup C \cup \{a,b,c\}=E(G)$. Now the following hold:
\begin{enumerate}[(a)]
\item Let $i,j,\ell \in \{1,\ldots,k\}$ with $j \neq \ell$ and $i \in \{j,\ell\}$, and let $(T_1,\ldots,T_k)$ be a factorization of $G$ that satisfies the following conditions:
\begin{itemize}
%\item {\color{red}$a \in E(T_p)$ if and only if $\sigma_A(a)=p$ for some bijection $\sigma_A:A \rightarrow  \{1,\ldots,k\}-i$,}
\item for $p \in \{1,\ldots,k\}-i$, $T_p$ contains exactly one edge of $A$,
\item $T_i$ contains $a$,
%\item {\color{red}$b \in E(T_p)$ if and only if $\sigma_B(b)=p$ for some bijection $\sigma_B:B \rightarrow  \{1,\ldots,k\}-j$,}
\item for $p \in \{1,\ldots,k\}-j$, $T_p$ contains exactly one edge of $B$,
\item $T_j$ contains $c$,
%\item {\color{red}$c \in E(T_p)$ if and only if $\sigma_C(c)=p$ for some bijection $\sigma_C:C \rightarrow  \{1,\ldots,k\}-\ell$,}
\item for $p \in \{1,\ldots,k\}-\ell$, $T_p$ contains exactly one edge of $C$,
\item $T_{\ell}$ contains $b$.
\end{itemize}
Then for $p=1,\ldots,k$, we have that $T_p$ is a spanning tree of $G$.
\item Let $(T_1,\ldots,T_k)$ be a spanning tree factorization of $G$ such that $|E(T_i)\cap (A \cup a)|=1$ for all $i \in \{1,\ldots,k\}$. Then there are some $j, \ell \in \{1,\ldots,k\}$ with the following properties: 
\end{enumerate} 
\begin{itemize}
\item $j \neq \ell$,
\item for $p \in \{1,\ldots,k\}-j$, $T_p$ contains exactly one edge of $B$,
\item $T_j$ contains $c$,
\item for $p \in \{1,\ldots,k\}-\ell$, $T_p$ contains exactly one edge of $C$,
\item $T_{\ell}$ contains $b$.
\end{itemize}
\end{Proposition}
\begin{proof}
$(a)$ can readily be checked. For $(b)$, let $j \in \{1,\ldots,k\}$ such that $E(T_j)$ contains $c$. 

\begin{Claim}$E(T_j)\cap B= \emptyset$.
\end{Claim}
\begin{proof}Suppose otherwise. Hence, as $T_j$ does not contain a triangle, we obtain that $E(T_j)\cap A= \emptyset$. By $|E(T_i)\cap (A \cup a)|=1$ for all $i \in \{1,\ldots,k\}$, we obtain that $a \in E(T_j)$. As $E(T_j)$ contains only 3 edges, we obtain that $E(T_j)\cap C= \emptyset$ and there is some $\ell \in \{1,\ldots,k\}-j$ such that $b \in E(T_{\ell})$. By $|A|=|C|=k-1, E(T_j)\cap A= \emptyset, E(T_j)\cap C= \emptyset$ and the fact that $(T_1,\ldots,T_k)$ is a set of spanning trees, we obtain that $E(T_{\ell})\cap A \neq \emptyset$ and $E(T_{\ell})\cap C \neq \emptyset$. This yields that $T_{\ell}$ contains a triangle, a contradiction. 
\end{proof}

As $B$ contains $k-1$ parallel edges and no spanning tree can contain two parallel edges, we obtain that  for $p \in \{1,\ldots,k\}-j$, $T_p$ contains exactly one edge of $B$. Similarly, there is some $\ell \in \{1,\ldots,k\}$ such that $T_{\ell}$ contains $b$ and for $p \in \{1,\ldots,k\}-\ell$, $T_p$ contains exactly one edge of $C$. Finally, suppose that $j = \ell$. As $T_j$ does not contain a triangle, there is some $m\in \{1,\ldots,k\}-j$ such that $a \in E(T_m)$. As  $E(T_j)\cap B= \emptyset$ and $E(T_j)\cap C= \emptyset$, we obtain that $E(T_m)\cap B\neq \emptyset$ and $E(T_m)\cap C\neq \emptyset$. This yields that $T_m$ contains a triangle, a contradiction.
\end{proof}
\section{Complexity}\label{complex}
The objective of this section is to prove the negative results we have on the complexity of several problems on packing common bases of a graphic matroid and a partition matroid. In Section \ref{redtrees}, we prove Theorem \ref{hard1} and in Section \ref{gziipj}, we prove Theorem \ref{hard2}.

\subsection{The proof of Theorem \ref{hard1}}\label{redtrees}

The purpose of this section is to show that RST$k$F is NP-hard. This proof is separated into three parts. In the first two parts, we consider the case $k=2$. In Section \ref{genk2}, we show that a slightly more general problem is NP-hard. In Section \ref{concfac}, we use the result of Section \ref{genk2} to finish the reduction for $k=2$.  In Section \ref{genk3}, we give the reduction for $k \geq 3$. %following problem which is slightly more general than RST$k$F

%In Section \ref{genk2}, we show that GRST$k$F is NP-hard for $k=2$ and in Section \ref{genk3}, we show that GRST$k$F is NP-hard for $k\geq 3$. Finally, in Section \ref{concfac}, we use these results to prove Theorem \ref{ggh}. 
\subsubsection{A slightly more general problem}\label{genk2}

As a preliminary step for proving that RST2F is hard, we here show the hardness of the following problem in which the condition on the partition is somewhat relaxed. 

%The next one will serve as an intermediate step in the proof of Theorem \ref{hard1}.
\medskip 

\noindent \textbf{Generalized Rainbow Spanning Tree 2-Factorization (GRST2F):}
\smallskip

\noindent\textbf{Input:} A double tree $G$, a 2-bounded partition $\mathcal{P}$ of $E(G)$.
\smallskip

\noindent\textbf{Question:} Can $G$ be factorized into 2 rainbow spanning trees?
\medskip
\begin{Lemma}\label{general}
GRST2F is NP-hard.
\end{Lemma}
\begin{proof}
We give a reduction from MNAE3SAT. Let $(X,\mathcal{C})$ be an instance of MNAE3SAT. Let $Z$ be the set of pairs $(x,C)$ such that $x \in C \in \mathcal{C}$. For every $(x,C) \in Z $, let $G^{(x,C)}$ be a copy of $K_4$ in which $\{e_1^{(x,C)},e_2^{(x,C)}\},\{f_1^{(x,C)},f_2^{(x,C)}\}$ and $\{g_1^{(x,C)},g_2^{(x,C)}\}$ are the 3 unique pairs of non-adjacent edges.  For every $C \in \mathcal{C} $, let $G^C$ be a copy of $K_4$ in which $\{e^C,f^C,g^C\}$ are 3 edges all incident to a common vertex. We now define $G$ by choosing an arbitrary vertex $v_M \in V(G^M)$ for all $M \in Z \cup \mathcal{C}$ and identifying all these vertices into a single vertex. We next define $\mathcal{P}$. First, for all $(x,C) \in Z$, let $\mathcal{P}$ contain $\{e_1^{(x,C)},e_2^{(x,C)}\}$. Next, for every $x \in X$, let $C_1,\ldots,C_{t_x}$ be an arbitrary ordering of the clauses containing $x$. For all $x \in X$ and $i=1,\ldots,{t_x}-1$, let $\mathcal{P}$ contain $\{f_1^{(x,C_i)},g_1^{(x,C_{i+1})}\}$. Finally, for $(x,C)\in Z$, we choose an edge $h \in \{e^C,f^C,g^C\}$ and add $\{h,f_2^{(x,C)}\}$ to $\mathcal{P}$. All remaining edges form one-element classes of $\mathcal{P}$. As $(X,\mathcal{C})$ is an instance of MNAE3SAT, we can choose $\mathcal{P}$ to be a 2-bounded partition of $E(G)$.

We now show that $(X,\mathcal{C})$ is a positive instance of MNAE3SAT if and only if $(G,\mathcal{P})$ is a positive instance of GRST2F.

First suppose that $(X,\mathcal{C})$ is a positive instance of MNAE3SAT, so there is a truth assignment $\Phi:X \rightarrow \{True,False\}$ such that every $C \in \mathcal{C}$ contains at least one true and one false literal. For every $(x,C) \in Z$ such that $\Phi(x)=true$, let $T_1^{(x,C)}$ be the spanning tree of $G^{(x,C)}$ whose edge set is $\{e_1^{(x,C)},f_1^{(x,C)},f_2^{(x,C)}\}$ and  let $T_2^{(x,C)}$ be the spanning tree of $G^{(x,C)}$ whose edge set is $\{e_2^{(x,C)},g_1^{(x,C)},g_2^{(x,C)}\}$.  For every $(x,C) \in Z$ such that $\Phi(x)=false$, let $T_1^{(x,C)}$ be the spanning tree of $G^{(x,C)}$ whose edge set is $\{e_1^{(x,C)},g_1^{(x,C)},g_2^{(x,C)}\}$ and  let $T_2^{(x,C)}$ be the spanning tree of $G^{(x,C)}$ whose edge set is $\{e_2^{(x,C)},f_1^{(x,C)},f_2^{(x,C)}\}$. Observe that in either case, $(T_1^{(x,C)},T_2^{(x,C)})$ is a spanning tree factorization of $G^{(x,C)}$.

Now consider some $C \in \mathcal{C}$. We create sets $E_1^C,E_2^C$ in the following way: Let $h \in \{e^C,f^C,g^C\}$. Observe that there is some unique $(x,C) \in Z$ such that $\{f_2^{(x,C)},h\} \in \mathcal{P}$. If $f_2^{(x,C)} \in E(T_1^{(x,C)})$, let $h$ be contained in $E_2^C$ and if $f_2^{(x,C)} \in E(T_2^{(x,C)})$, let $h$ be contained in $E_1^C$. 

\begin{Claim}
Both $E_1^C$ and $E_2^C$ are nonempty. 
\end{Claim}
\begin{proof}
As $\Phi$ satisfies $(X,\mathcal{C})$, there is some $x \in X$ such that $(x,C)\in Z$ and $\Phi(x)=true$. It follows that $f_2^{(x,C)}\in E(T_1^{(x,C)})$. Now $\mathcal{P}$ contains $\{f_2^{(x,C)},h\}$ for some $h \in \{e^C,f^C,g^C\}$, so $h \in E_2^C$. It follows that $E_2^C$ is nonempty. Similarly, the fact that there is some $x \in X$ such that $(x,C)\in Z$ and $\Phi(x)=false$ implies that $E_1^C$ is nonempty.
\end{proof}

By Proposition \ref{fzf}, we obtain that there is a spanning tree factorization $(T_1^C,T_2^C)$ of $G^C$ such that $E_1^C\subseteq E(T_1^C)$ and $E_2^C\subseteq E(T_2^C)$. Now, for $i=1,2$, let $T_i$ be obtained from the union of $\{T_i^M|M \in \mathcal{C} \cup Z\}$ by identifying the vertices $v_M$ for all $M \in \mathcal{C}\cup Z$. By Proposition \ref{ident}, $(T_1,T_2)$ is a spanning tree factorization of $G$. Further, by construction, both of $T_1$ and $T_2$ are rainbow.

Now suppose that $(G,\mathcal{P})$ is a positive instance of GRST2F, so there is a factorization of $G$ into two rainbow spanning trees $(T_1,T_2)$. For $(x,C)\in Z$ and $i \in \{1,2\}$, let $T_i^{(x,C)}$ be the restriction of $T_i$ to $G^{(x,C)}$. By Proposition \ref{ident}, $(T_1^{(x,C)}, T_2^{(x,C)})$ is a spanning tree factorization of $G^{(x,C)}$.

\begin{Claim}
For every $x \in X$, there is some $i \in \{1,2\}$ such that $f_2^{(x,C)}\in E(T_i)$ for all $C \in \mathcal{C}$ with $(x,C) \in Z$.
\end{Claim}
\begin{proof}Let $C_1,\ldots,C_{t_x}$ be the ordering of the clauses containing $x$ that was used in the construction of $(G,\mathcal{P})$. By symmetry and induction, it suffices to prove that if $f_2^{(x,C_j)}\in E(T_1)$, then $f_2^{(x,C_{j+1})}\in E(T_1)$ for all $j\in 1,\ldots,t_x-1$. Fix some $j\in 1,\ldots,t_x-1$ and suppose that $f_2^{(x,C_j)}\in E(T_1)$. As $(T_1^{(x,C_j)}, T_2^{(x,C_j)})$ is a spanning tree factorization of $G^{(x,C_j)}$ and $\{e_1^{(x,C_j)},e_2^{(x,C_j)}\} \in \mathcal{P}$, we obtain by Proposition \ref{long} $(b)$ that $f_1^{(x,C_j)}\in E(T_1)$. As $\{f_1^{(x,C_j)},g_1^{(x,C_{j+1})}\} \in \mathcal{P}$, we obtain that $g_1^{(x,C_{j+1})}\in E(T_2)$. Hence, as $(T_1^{(x,C_{j+1})}, T_2^{(x,C_{j+1})})$ is a spanning tree factorization of $G^{(x,C_{j+1})}$ and $\{e_1^{(x,C_{j+1})},e_2^{(x,C_{j+1})}\} \in \mathcal{P}$, we obtain by Proposition \ref{long} $(b)$ that $f_2^{(x,C_{j+1})}\in E(T_1)$.
\end{proof}

We may now define a truth assignment $\Phi:X \rightarrow \{true,false\}$ for $X$ in the following way: We let $\Phi(x)=true$ if $f_2^{(x,C)}\in E(T_1)$ for all $C \in \mathcal{C}$ with $(x,C) \in Z$ and $\Phi(x)=false$ if $f_2^{(x,C)}\in E(T_2)$ for all $C \in \mathcal{C}$ with $(x,C) \in Z$. 

\begin{Claim}
For every $C \in \mathcal{C}$, there is a variable $x \in X$ with $(x,C) \in Z$ and $\Phi(x)=true$.
\end{Claim}
\begin{proof}
By Proposition \ref{ident}, $(T_1[V(G^C)],T_2[V(G^C)])$ is a spanning tree factorization of $G^C$. Hence by Proposition \ref{fzf}, there is some $h \in \{e^C,f^C,g^C\}$ such that $h \in E(T_2[V(G^C)])\subseteq E(T_2)$. There is a unique variable $x \in X$ such that $(x,C)\in Z$ and $\{f_2^{(x,C)},h\} \in \mathcal{P}$. As $(T_1,T_2)$ is a factorization into rainbow spanning trees of $G$ with respect to $\mathcal{P}$, we obtain that $f_2^{(x,C)} \in E(T_1)$. We obtain by construction that $\Phi(x)=true$. 
\end{proof}

Similarly, there is a variable $x \in X$ with $(x,C) \in Z$ and $\Phi(x)=false$. As $C$ was chosen arbitrarily, we obtain that $(X,\mathcal{C})$ is a positive instance of MNAE3SAT. This finishes the proof.
\end{proof}

\subsubsection{The reduction of RST2F}\label{concfac}

In this section we prove the NP-completeness of RST2F through a reduction from GRST2F. Let $(G,\mathcal{P})$ be an instance of GRST2F. We now create an instance $(G',\mathcal{P}')$ of RST2F. We first create 2 copies $G_1,G_2$ of $G$. For every $v \in V(G)$ and $i \in \{1,2\}$, we denote by $v_i$ the copy of $v$ in $G_i$ and for every $e \in E(G)$  and $i \in \{1,2\}$, we denote by $e_i$ the copy of $e$ in $G_i$. We obtain $G'$ by choosing some arbitrary $v \in V(G)$ and identifying $\{v_1,v_2\}$ into a single vertex. Next for all $X \in \mathcal{P}$ and $i \in\{1,2\}$ with $| X|=2$, we let $\mathcal{P}'$ contain the partition class $X_i=\{e_i:e \in X\}$. Further, for every $e \in E(G)$ which is a one element class of $\mathcal{P}$, we let $\mathcal{P}'$ contain the partition class $Y_e=\{e_1,e_2\}$. Observe that $G'$ is a double tree by Proposition \ref{ident} and that $\mathcal{P}'$ is a 2-uniform partition of $E(G')$. We now show that $(G',\mathcal{P}')$ is a positive instance of RST2F if and only if $(G,\mathcal{P})$ is a positive instance of GRST2F.

First suppose that $(G',\mathcal{P}')$ is a positive instance of RST2F, so there is a factorization $(T'_1,T'_2)$ of $G'$ such that $T'_i$ is rainbow with respect to $\mathcal{P}'$ for $i=1,2$. For $i=1,2$, let $T_i$ be the spanning subgraph of $G$ whose edge set is $\{e \in E(G)|e_1 \in E(T'_i)\}$. As $T'_i[V(G_1)]$ is a spanning tree of $G_1$ by Proposition \ref{ident}, we obtain that $T_i$ is a spanning tree of $G$. Also, for any $X \in \mathcal{P}$, we have that $|X_1 \cap E(T'_i)|\leq 1$, hence $|X \cap E(T_i)|\leq 1$, so $T_i$ is rainbow with respect to $\mathcal{P}$. Finally, as $T_1'$ and $T_2'$ are edge-disjoint , so are $T_1$ and $T_2$. Hence $(T_1,T_2)$ is a factorization of $G$ into spanning trees which are rainbow with respect to $\mathcal{P}$, so $(G,\mathcal{P})$ is a positive instance of GRST2F.

Now suppose that $(G,\mathcal{P})$ is a positive instance of GRST2F, so there is a factorization $(T_1,T_2)$ of $G$ into spanning trees which are rainbow with respect to $\mathcal{P}$. % For any $e \in E(G)$, let $\alpha_e$ be the unique integer such that $e \in E(T_{\alpha_e})$. 
 For $j=1,2$, let $T_j'$ be the spanning subgraph of $G'$ with edge set $$\{e_1:e \in E(T_j)\}\cup \{e_2: e \in E(T_{3-j})\}.$$
%For any $j=1,\ldots,k$, let $T'_j$ be the unique spanning subgraph of $G'$ whose edge set is $\{e_i|i+\alpha_e\equiv j \mod k, e \in E(G),i=1,\ldots,k\} $.
 It follows straight from the definition that $T'_1$ and $T'_{2}$ are edge-disjoint. For all $i,j \in \{1,2\}$, we have that $T'_j[V(G_i)]$ is a spanning tree of $G_i$. It follows from Proposition \ref{ident} that $T'_j$ is a spanning tree of $G'$. We still need to show that $T'_j$ is rainbow with respect to $\mathcal{P}'$. By symmetry, it suffices to prove that $T'_1$ is rainbow with respect to $\mathcal{P}'$. For every $X \in \mathcal{P}$, we have $|E(T'_1)\cap X_1|=|E(T_1)\cap X|\leq 1$. Similary, we have $|E(T'_1)\cap X_2| \leq 1$ for all $X \in \mathcal{P}$. For all $e \in E(G)-\bigcup \mathcal{P}$, we have $|E(T'_j)\cap Y_e|=1$ by construction. Hence $T'_1$ is rainbow with respect to $\mathcal{P}'$. It follows that $(G',\mathcal{P}')$ is a positive instance of RST2F.

By Lemma \ref{general}, we obtain that Theorem \ref{hard1} holds for $k=2$.

\subsubsection{The reduction of RST$k$F for $k\geq 3$}\label{genk3}
We fix some $k \geq 3$ and show that RST$k$F is NP-hard using a reduction from $k$COL. Let $H$ be a graph. We now create an instance $(G,\mathcal{P})$ of RST$k$F which is positive if and only if $H$ is a positive instance of $k$COL. For every $e=uv \in E(H)$, we let $G$ contain a $(k-1)$-marihuana leaf $G^{e}$. We denote a set of $k-1$ parallel edges in $E(G^{e})$ by $F^{e}$ and the unique edge in $E(G^{e})$ which is disjoint from all edges in $F^{e}$ by $f^{e}$. Further, we choose a set of $k-1$ parallel edges in $E(G^{e})$ which is different from $F^{e}$ and denote it by $F^{e}_u$. The unique set of $k-1$ parallel edges in $G^{e}$ which is different from both $F^{e}$ and $F^{e}_u$ is denoted by $F^{e}_v$. Finally we denote the unique edge in $E(G^{e})-(F^{e}\cup f_e \cup F^{e}_u \cup F^{e}_v)$ which shares one vertex with the edges in $F^{e}_u$ by $f^{e}_u$ and the unique edge in $E(G^{e})-(F^{e}\cup f_e \cup F^{e}_u \cup F^{e}_v)$ which shares one vertex with the edges in $F^{e}_v$ by $f^{e}_v$. An illustration can be found in Figure \ref{fig1}.

\begin{figure}[h]
    \centering
        \includegraphics[width=.5\textwidth]{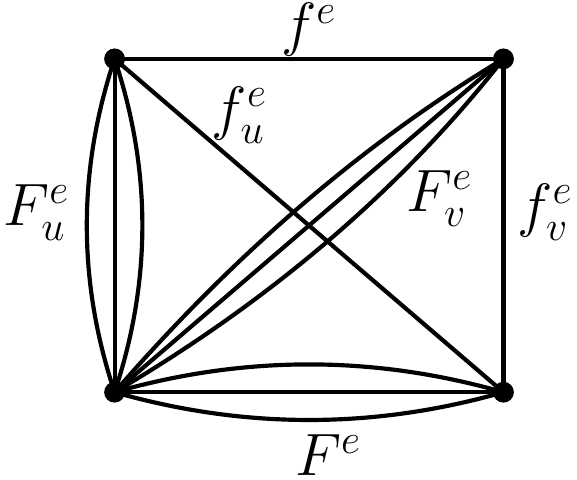}
        \caption{A 3-marihuana leaf whose edges are labeled as in Section \ref{genk3}.}\label{fig1}
\end{figure}

We now obtain $G$ by picking an arbitrary vertex $w_e \in V(G^{e})$ for all $e \in E(H)$ and identifying all the vertices $w_e$ into a single vertex. By Proposition \ref{long} $(a)$ and Proposition \ref{ident}, we obtain that $G$ is a $k$-multiple tree. In order to define $\mathcal{P}$, we first add $F^{e}\cup f^{e}$ to $\mathcal{P}$ for all $e \in E(H)$. Next, for every $v \in V(H)$, let $e_1,\ldots,e_{t_v}$ be an arbitrary ordering of the edges $v$ is incident to in $H$. For every $i=1,\ldots,t_v-1$, we add $F^{e_i}_v\cup f^{e_{i+1}}_v$ to $\mathcal{P}$. We also add $F^{e_{t_v}}_v\cup f^{e_{1}}_v$ to $\mathcal{P}$. Observe that $\mathcal{P}$ is a $k$-uniform partition of $E(G)$. This finishes the definition of $(G,\mathcal{P})$. We now show that $(G,\mathcal{P})$ is a positive instance of RST$k$F if and only if $H$ is a positive instance of $k$COL.

First suppose that $(G,\mathcal{P})$ is a positive instance of RST$k$F, so there is a factorization of $G$ into $k$ spanning trees $(T_1,\ldots,T_k)$ which are rainbow with respect to $\mathcal{P}$. For every $e \in E(H)$ and $i \in \{1,\ldots,k\}$, let $T_i^{e}$ be the restriction of $T_i$ to $V(G_e)$. By Proposition \ref{ident}, we obtain that $(T_1^{e},\ldots,T_k^{e})$ is a spanning tree factorization of $G^{e}$ for all $e \in E(H)$.

\begin{Claim}\label{ds}
For every $v \in V(H)$, there is some $i \in \{1,\ldots,k\}$ such that $f^{e}_v \in E(T_i)$ for all $e \in\delta_H(v)$.
\end{Claim}
\begin{proof}
Let $e_1,\ldots,e_{t_v}$ be the ordering of the edges incident to $v$ used in the construction of $(G,\mathcal{P})$. It suffices to prove that if $f^{e_j}_v\in E(T_i)$ for some $j \in \{1,\ldots t_v-1\}$ and $i \in \{1,\ldots,k\}$, then $f^{e_{j+1}}_v\in E(T_i)$. As $(T_1^{e_j},\ldots,T_k^{e_j})$ is a spanning tree factorization of $G^{e_j}, F^{e_j}\cup f^{e_j} \in \mathcal{P}$, and by Proposition \ref{long} $(b)$, we obtain that $|F^{e_j}_v \cap E(T_{\ell})|=1$ for all $\ell \in \{1,\ldots,k\}-i$ and $|F^{e_j}_v \cap E(T_i)|=0$. As $F^{e_j}_v\cup f^{e_{j+1}}_v \in \mathcal{P}$, we obtain that $f^{e_{j+1}}_v\in E(T_i)$.
\end{proof}

We now define a coloring $\phi:V(H) \rightarrow \{1,\ldots,k\}$ in the following way: we set $\phi(v)=j$ if $f^{e}_v\in E(T_j)$ for all $e \in \delta_H(v)$. Observe that $\phi$ is well-defined by Claim \ref{ds}. In order to show that $\phi$ is a proper vertex-coloring of $H$, consider some $e=uv \in E(H)$. As $(T_1^{e},\ldots,T_k^{e})$ is a spanning tree factorization of $G^{e}$ and by Proposition \ref{long} $(b)$, we obtain that there are some $i_u,i_v \in \{1,\ldots,k\}$ such that $i_u \neq i_v, f^{e}_u \in E(T_{i_u})$ and $f^{e}_v \in E(T_{i_v})$. We obtain by construction that $\phi(u)=i_u$ and  $\phi(v)=i_v$, so in particular $\phi(u)\neq \phi(v)$. It follows that $\phi$ is a proper vertex-coloring of $H$, so $H$ is a positive instance of $kCOL$.

Now suppose that there exists a proper vertex-coloring $\phi:V(H) \rightarrow \{1,\ldots,k\}$ of $H$.  We next define a spanning tree factorization of $G$. First fix some $e=uv \in E(H)$. We define a spanning tree factorization $(T^{e}_1,\ldots,T^{e}_k)$ of $G^{e}$ in the following way: We let $F^{e}$ contain an edge of $E(T^{e}_i)$ for all $i \in \{1,\ldots,k\}-\phi(u)$ and we let $f^{e}$ be contained in $E(T^{e}_{\phi(u)})$. Further, we let $F^{e}_u$ contain an edge of $E(T^{e}_i)$ for all $i \in \{1,\ldots,k\}-\phi(u)$ and we let $f^{e}_u$ be contained in $E(T^{e}_{\phi(u)})$. Finally, we let $F^{e}_v$ contain an edge of $E(T^{e}_i)$ for all $i \in \{1,\ldots,k\}-\phi(v)$ and we let $f^{e}_v$ be contained in $E(T^{e}_{\phi(v)})$. As $\phi$ is a proper vertex-coloring, we have $\phi(u)\neq \phi(v)$ and hence by  Proposition \ref{long} $(a)$, we obtain that $(T^{e}_1,\ldots,T^{e}_k)$ is a spanning tree factorization of $G^{e}$. For all $i \in \{1,\ldots,k\}$, we now obtain $T_i$ from the union of $T^{e}_i$ for all $e \in E(H)$ by identifying the vertices $w_e$. By Proposition \ref{ident}, we obtain that $(T_1,\ldots,T_k)$ is a spanning tree factorization of $G$. Further, by construction, $T_i$ is rainbow with respect to $\mathcal{P}$ for all $i=1,\ldots,k$. Hence, $(T_1,\ldots,T_k)$ is a factorization of $G$ into rainbow spanning trees, so $(G,\mathcal{P})$ is a positive instance of RST$k$F.

 We hence obtain that Theorem \ref{hard1} holds for $k \geq 3$. Together with the results in Section \ref{concfac}, this finishes the proof of Theorem \ref{hard1}.

\subsection{Application to $g$-bounded spanning trees}\label{gziipj}

In this section, we prove Theorem \ref{hard2}. % Formally, we consider the following decision problem:
% Then a spanning tree $T$ of $G$ is called $g$-bounded for some given function $g:V(G)\rightarrow \mathbb{Z}_+$ if $d_{\vec{E}}^-(v)\leq g(v)$ for all $v \in V$.
 We use Theorem \ref{hard1}.

\begin{proof}[Proof of Theorem \ref{hard2}.]
We prove that BST$k$F is NP-hard by a reduction from RST$k$F. Let $(G,\mathcal{P})$ be an instance of RST$k$F. %Our proof works in two steps. Before creating the final instance of  BST$k$F, we deal with a similar problem in mixed graphs. 

We create a digraph $D$ in the following way: First, we let $V(D)$ contain the vertices in $V(G)$, a vertex $w_e$ for every $e \in E(G)$, and a set $Z_X$ of $k-1$ vertices for every $X \in \mathcal{P}$. Let $Z=\bigcup_{X \in \mathcal{P}}Z_X$. %We first let $E(H)$ contain all the edges of $E(G)-\bigcup \mathcal{P}$.
  Now, for every $e \in E(G)$, let $\vec{e}$ be an arbitrary orientation of $e$. For every $e \in E(G)$ with $\vec{e}=uv$, we let $A(D)$ contain $k$ parallel arcs from $u$ to $w_e$ and one arc from $v$ to $w_e$. Further, for every $X \in \mathcal{P}$, we let $A(D)$ contain an arc from every vertex in $Z_X$ to every vertex in $B_X=\{w_e|e \in X\}$.

Finally, we let $g:V(D)\rightarrow \mathbb{Z}_{\geq 0}$ be defined by $g(v)=0$ for all $v \in V(G)\cup Z$ and $g(w_e)=2$ for all $e \in E(G)$. 
\medskip

%\begin{Lemma}\label{efeq}
%The following two statements are equivalent:
%\begin{itemize}
%\item $(G,\mathcal{P})$ is a positive instance of RST$k$F,
%\item $H$ contains a factorization into subgraphs $(U_1,\ldots,U_k)$ such that $UG(U_i)$ is a spanning tree of $UG(H)$ and $d_{U_i}^-(v)\leq 1$ for all $i \in \{1,\ldots,k\}$ and $v \in V(H)$.
%\end{itemize}
%\end{Lemma}

We now show that $(G,\mathcal{P})$ is a positive instance of RST$k$F if and only if $(D,g)$ is a positive instance of  BST$k$F.
\medskip

First suppose that $(G,\mathcal{P})$ is a positive instance of RST$k$F, so there is a factorization $(T_1,\ldots,T_k)$ of $G$ into rainbow spanning trees. We now create a set of spanning subgraphs $(U_1,\ldots,U_k)$ of $D$. %First, for every $i \in \{1,\ldots,k\}$ and $e \in E(T_i)\cap (E(G)-\bigcup \mathcal{P})$, we let $E(U_i)$ contain $e$.
First, for every $i \in \{1,\ldots,k\}$ and $e\in E(T_i)$ with $\vec{e}=uv$, we let $A(U_i)$ contain the arc $vw_e$ and we let $A(U_j)$ contain one of the arcs from $u$ to $w_e$ for all $j \in \{1,\ldots,k\}$.

Now for some $X \in \mathcal{P}$ consider the complete bipartite graph $H_X$ whose partition classes are $Z_X$ and $B_X$. Further, consider the mapping $\sigma_X:B_X\rightarrow \{1,\ldots,k\}$ where $\sigma_X(w_e)$ is the unique $i \in \{1,\ldots,k\}$ such that $e \in E(T_i)$. Observe that, as $(T_1,\ldots,T_k)$ is a factorization of $G$ into spanning trees which are rainbow with respect to $\mathcal{P}$ and $\mathcal{P}$ is $k$-uniform, for every $i \in \{1,\ldots,k\}$, there is exactly one $w_e \in B_X$ such that $\sigma_X(w_e)=i$. Hence, by Proposition \ref{compbip}, there is a function $\phi_X:E(H_X)\rightarrow \{1,\ldots,k\}$ such that $\phi_X(\delta_{H_X}(z))=\{1,\ldots,k\}$ for all $z \in Z_X$ and $\phi_X(\delta_{H_X}(w_e))=\{1,\ldots,k\}-\sigma_X(w_e)$ for all $w_e \in B_X$. Now, for every $i \in \{1,\ldots,k\}$, for every $X \in \mathcal{P}$, and for every arc $a$ from $Z_X$ to $B_X$, we add $a$ to $A(U_i)$ if $\phi_X(e)=i$ where $e$ is the edge in $E(H_X)$ corresponding to $a$. This finishes the definition of $(U_1,\ldots,U_k)$. Observe that $(U_1,\ldots,U_k)$ is a factorization of $D$ into subgraphs.

In order to see that $UG(U_i)$ is a spanning tree of $UG(D)$ for all $i \in \{1,\ldots,k\}$, first observe that all vertices in $Z$ are of degree 1 in $UG(U_i)$, all the vertices in $\{w_e|e \in E(G)-E(T_i)\}$ are of degree 1 in $UG(U_i)-Z$, and all the vertices in $\{w_e|e \in E(T_i)\}$ are of degree 2 in $UG(U_i)-Z$. Further, the graph obtained from $UG(U_i)$ by deleting the vertices in $Z \cup \{w_e|e \in E(G)-E(T_i)\}$ and then suppressing the vertices in $\{w_e|e \in E(T_i)\}$ is exactly $T_i$, which is a spanning tree of $G$. It follows that $UG(U_i)$ is a spanning tree of $UG(H)$.

For all $v \in V(G)\cup Z$, we have $d_{D}^-(v)=0$, so in particular $d_{U_i}^-(v)=0$ for all $i \in \{1,\ldots,k\}$. Now consider some $e \in E(G)$ with $\vec{e}=uv$ and $e\in X$ for some $X \in \mathcal{P}$ and some $i \in \{1,\ldots,k\}$. By construction, exactly one of the arcs from $u$ to $w_e$ enters $w_e$.  If $e \in E(T_i)$, we obtain that $vw_e \in A(U_i)$ and so $\sigma_X(w_e)=i$. We hence obtain that $\phi_X(zw_e)\neq i$ for all $z \in Z_X$ and so $A(U_i)\cap \delta_H(Z_X,w_e)=\emptyset$, so $d_{U_i}^-(w_e)=2$. Otherwise, we have $vw_e \in A(D)-A(U_i)$ and therefore $$d_{U_i}^-(w_e)=d_{U_i}(Z_x,w_e)+d_{U_i}(u,w_e)=|\{f \in \delta_{H_X}(w_e)|\phi_X(f)=i\}|+1=2.$$ Hence $(U_1,\ldots,U_k)$ has the desired properties.
\medskip

Now suppose that  $(D,g)$ is a positive instance of  BST$k$F, so $D$ contains a factorization into subgraphs $(U_1,\ldots,U_k)$ such that $UG(U_i)$ is a spanning tree of $UG(H)$ and $d_{U_i}^-(v)\leq g(v)$ for all $i \in \{1,\ldots,k\}$ and $v \in V(D)$. We define a set of spanning subgraphs $(T_1,\ldots,T_k)$ of $G$ in the following way. For every $i \in \{1,\ldots,k\}$ and every $e \in E(G)$ with $\vec{e}=uv$ such that $vw_e \in A(U_i)$, we let $E(T_i)$ contain $e$.

It is easy to see that $(T_1,\ldots,T_k)$ is a factorization of $G$. We now fix some $i \in\{1,\ldots,k\}$ and prove that $T_i$ is a spanning tree of $G$. First observe that $d_{UG(U_i)}(z)=1$ for all $z \in Z$ because $d_{UG(D)}(z)=k$ and $UG(U_j)$ is a spanning tree of $UG(D)$ for all $j \in \{1,\ldots,k\}$. Next observe that all the vertices in $\{w_e|e \in E(G)-E(T_i)\}$ are of degree 1 in $UG(U_i)-Z$ and all the vertices in $\{w_e|e \in E(T_i)\}$ are of degree 2 in $UG(U_i)-Z$ because being a spanning tree of $UG(D)$, $UG(U_j)$ has to contain exactly one of the $k$ edges corresponding to the parallel arcs from $u$ to $w_e$ for all $j \in \{1,\ldots,k\}$ and $e \in E(G)$ with $\vec{e}=uv$. Further, the graph obtained from $UG(U_i)$ by deleting the vertices in $Z \cup \{w_e: e \in E(G)-E(T_i)\}$ and then suppressing the vertices in $\bigcup_{e \in E(T_i)}w_e$ is exactly $T_i$. As $UG(U_i)$ is a spanning tree of $UG(D)$, we obtain that $T_i$ is a spanning tree of $G$. 

We still need to prove that $T_i$ is rainbow with respect to $\mathcal{P}$ for all $i \in \{1,\ldots,k\}$. Fix some $i \in \{1,\ldots,k\}$ and $X \in \mathcal{P}$.
 Every $z \in Z_X$ is incident to exactly $k$ edges in $UG(D)$. As $UG(U_j)$ is a spanning tree of $UG(D)$ for every $j \in \{1,\ldots,k\}$, we obtain that every $z \in Z_X$ is incident to exactly one edge in $UG(U_i)$. As $|Z_X|=k-1$ and $|X|=k-1$, there is exactly one $e \in X$ such that $U_i$ does not contain an arc from $Z_X$ to $w_e$. This yields $|E(T_i)\cap X|=|\{e\}|=1$.
% By construction, we have $|E(T_i)\cap X|=d_{U_i}^-(B_X)-d_{U_i}(Z_X,B_X)-|B_X|$. As $|Z_X|=k-1$, $d_{UG(D)}(z)=k$ for all $z \in Z_X$ and $UG(U_j)$ is a spanning tree of $UG(D)$ for all $j \in \{1,\ldots,k\}$, we obtain that $d_{U_i}(Z_X,B_X)=d_{UG(U_i)}(Z_X,B_X)=|Z_X|=k-1$. Further, as $d_{D}^-(w_e)=2k$ and $d_{U_j}^-(w_e)\leq 2$ for all $j \in \{1,\ldots,k\}$ and for all $e \in X$, we have $d_{U_i}^-(w_e)=2$ for all $e \in X$. This yields $d_{U_i}^-(B_X)=|B_X|=k$. We obtain $|E(T_i)\cap X|=d_{U_i}^-(B_X)-d_{U_i}(Z_X,B_X)=k-(k-1)=1$, so $T_i$ is rainbow with respect to $\mathcal{P}$ indeed.
 Hence $(G,\mathcal{P})$ is a positive instance of RST$k$F.
\medskip

We obtain that $(G,\mathcal{P})$ is a positive instance of RST$k$F if and only if $(D,g)$ is a positive instance of BST$k$F. As $k$ is a fixed number, we further obtain that the size of $(D,g)$ is polynomial in the size of $(G,\mathcal{P})$. Hence the proof is finished by Theorem \ref{hard1}.
\end{proof}

\section{Covering by rainbow bases}\label{cover}

The objective of this section is to prove Theorems \ref{k41} to \ref{doublecover}. In Section \ref{ind}, we give a preliminary result on covering matroids with a given 2-uniform partition of their elements by a collection of rainbow independent sets. While this result may also be of independent interest, its main purpose in this article is to be used to prove Theorems \ref{k41} and \ref{k412} in Section \ref{aumoins3}. Finally, we prove Theorem \ref{doublecover} in Section \ref{k2}.
\subsection{Covering by independent sets}\label{ind}

This section contains a simple result on covering the elements of matroids by rainbow bases. %Its proof is pretty simple.
%{\color{red}Ergebnisse aus deiner Arbeit mit Thomas Kaiser}

We need the following preliminary result. It can easily be concluded from Edmonds' matroid intersection theorem \cite{e4} and also directly follows from Lemma \ref{main} in Section \ref{k2}. We leave the proof as an exercise.

\begin{Proposition}\label{dfqdqw}
Let $M$ be a 2-base matroid and $\mathcal{P}$ be a 2-uniform partition of $E(M)$. Then $M$ contains a rainbow basis.
\end{Proposition}
We are now ready to prove the main result of this section which is a consequence of Proposition \ref{dfqdqw}.

\begin{Lemma}\label{fzu}
Let $M$ be a 2-base matroid and $\mathcal{P}$ a 2-uniform partition of $E(M)$. Then $M$ can be covered by $3$ rainbow independent sets one of which is a basis.
\end{Lemma}
\begin{proof}
By assumption, $E(M)$ can be partitioned into two bases $B_1$ and $B_2$. % Next, let $\mathcal{P}'$ be a partition of $E(M)$ such that every class of $\mathcal{P}$ is also a class of $\mathcal{P}'$ and every class of $\mathcal{P}'$ is of size exactly 2. As $\mathcal{P}$ is 2-bounded and  $| E(M)|$ is even, we obtain that $\mathcal{P}'$ is well-defined.
 By Proposition \ref{dfqdqw}, $M$ contains a basis $X_1$ which is rainbow. Consider $X_2=(E(M)-X_1)\cap B_1$ and $X_3=(E(M)-X_1)\cap B_2$. Observe that $X_2$ and $X_3$ are independent in $M$. Further, for any $Y \in \mathcal{P
}$ and $i \in \{2,3\}$, we have $$| X_i \cap Y|\leq | (E(M)-X_1) \cap Y|=1.$$ It follows that  $X_1,X_2$ and $X_3$ are rainbow.
\end{proof}

 We wish to remark here that a rather technical degeneracy argument can be used to show that every $3$-multiple tree can be covered by 4 rainbow independent sets. A similar result for general matroids has been obtained by Aharoni, Berger and Ziv in \cite{ABZ}. Together with the results in Section \ref{aumoins3}, we obtain that the constant in Theorems \ref{k41} and \ref{k412} can be improved by 1 if $k$ is odd. Since the improvement is rather marginal and we wish to keep the proof self-contained, we omit it.

%We next show the following slightly stronger result for graphic matroids. 

%\begin{Theorem}
%Let $G$ be a double tree and $\mathcal{P}$ a 2-bounded subpartition of $E(G)$. Then $E(G)$ can be covered by $3$ rainbow forests with respect to $\mathcal{P}$.
%\end{Theorem}

%\begin{Theorem}
%Let $G$ be a 3-multiple tree and $\mathcal{P}$ a 2-bounded subpartition of $E(G)$. Then $E(G)$ can be covered by $4$ rainbow forests with respect to $\mathcal{P}$.
%\end{Theorem}

%\begin{Corollary}
%Let $M$ be a $k$-base matroid where $k=2 \alpha+ \beta$ for some nonnegative integers $\alpha$ and $\beta$ and $\phi$ a 2-bounded coloring of $E(M)$. Then $E(M)$ can be covered by $3 \alpha+2\beta$ independent sets of $M$ which are rainbow  with respect to $\phi$.
%\end{Corollary}

%\begin{Corollary}
%Let $M$ be a graphic $k$-base matroid where $k=2 \alpha+3 \beta$ for some nonnegative integers $\alpha$ and $\beta$ and $\phi$ a 2-bounded coloring of $E(M)$. Then $E(M)$ can be covered by $3 \alpha+4\beta$ independent sets of $M$ which are rainbow  with respect to $\phi$.
%\end{Corollary}

%\begin{Corollary}\label{kaiser1}
%Let $G$ be a $2t$-multiple tree for some positive integer $t$ and $\mathcal{P}$ a 2-bounded subpartition of $E(G)$. Then $E(G)$ can be covered by $3t$ rainbow forests with respect to $\mathcal{P}$.
%\end{Corollary}

%\begin{Corollary}\label{kaiser2}
%Let $G$ be a $3t$-multiple tree for some positive integer $t$ and $\mathcal{P}$ a 2-bounded subpartition of $E(G)$. Then $E(G)$ can be covered by $4t$ rainbow forests with respect to $\mathcal{P}$.
%\end{Corollary}
\subsection{Covering by rainbow bases for $k\geq 3$}\label{aumoins3}

In this section, we prove Theorems \ref{k41} and \ref{k412}.
We need the following two results which are basic matroid theory and hence given without proof.
\begin{Proposition}\label{trivmat1}
Let $M$ be a matroid, $B$ a basis of $M$ and $X \subseteq E(M)$. Then $B-X$ contains a basis of $M/X$.
\end{Proposition}

\begin{Proposition}\label{trivmat2}
Let $M$ be a matroid, $X$ a subset of $E(M)$ which is independent in $M$ and $B$ a basis of $M/X$. Then $B\cup X$ is a basis of $M$.
\end{Proposition}

The following two results are the key ingredient for relating the problem of covering by rainbow independent sets to the problem of covering by rainbow bases. We first give the following result that holds for matroids containing at least 4 disjoint bases.
\begin{Lemma}\label{foresttree}
Let $M$ be a matroid that contains 4 disjoint bases, $\mathcal{P}$ a 2-uniform partition of $E(M)$ and $X \subseteq E(M)$ a set which is independent in $M$ and rainbow. Then there are two rainbow bases $B_1,B_2$ of $M$ such that $X\subseteq B_1 \cup B_2$.
\end{Lemma}
%{\color{red}also works for matroids}
\begin{proof}
For all $\{x,y\}\in \mathcal{P}$, we say that $x$ and $y$ are {\it partners} in $\mathcal{P}$. Let $\bar{X}$ be the set of elements in $E(M)$ whose partners in $\mathcal{P}$ are in $X$ and let $\{C_1,\ldots,C_4\}$ be a set of disjoint bases of $M$. Now consider $\bar{X}_1=\bar{X}-(C_1\cup C_2)$ and $\bar{X}_2=\bar{X}-(C_3\cup C_4)$ and for $i=1,2$, let $X_i$ be the set of elements of $X$ whose partners are in $\bar{X}_i$. Finally consider the matroids $M_1=M|(X_1 \cup C_1 \cup C_2)$ and $M_2=M|(X_2 \cup C_3 \cup C_4)$.

\begin{Claim}
For $i=1,2$, $M_i$ has a rainbow basis $B_i$ such that $X_i \subseteq B_i$.
\end{Claim}
\begin{proof} By symmetry, it suffices to prove the statement for $i=1$. 
Consider the matroid $M_1'$ which is obtained from $M_1$ by contracting $X_1$. As $C_1$ and $C_2$ are bases of $M_1$,  we obtain by Proposition \ref{trivmat1} that $C_j$ contains a basis $C'_j$ of $M_1'$ for $j=1,2$. Let $M_1''=M_1'|(C_1' \cup C_2')$ and observe that $M_1''$ is a 2-base matroid. It hence follows from Proposition \ref{dfqdqw} that $M_1''$ has a rainbow basis $B_1'$. As $M_1''$ is a restriction of $M_1'$, we obtain that $B_1'$ is also a basis of $M_1'$. As $X_1$ is independent in $M$, we obtain that $B'_1\cup X_1$ is a basis of $M_1$ by Proposition \ref{trivmat2}. As $X_1$ is rainbow  by assumption, none of the partners of the edges of $X_1$ in $\mathcal{P}$ are in $E(M_1)$ and $B_1'$ is rainbow, it follows that $B'_1\cup X_1$ is rainbow.
\end{proof}
We obtain $X=X_1 \cup X_2\subseteq B_1 \cup B_2$. As $E(M_i)$ contains a basis of $M$ for $i=1,2$, we obtain that $B_1$ and $B_2$ are also bases of $M$. This finishes the proof.
\end{proof}

While Lemma \ref{foresttree} does not hold for matroids containing fewer than 4 bases, we have the following slightly weaker result for matroids containing 3 disjoint bases. It can be proven analogously to Lemma \ref{foresttree}.
\begin{Lemma}\label{foresttree2}
Let $M$ be a matroid that contains 3 disjoint bases, $\mathcal{P}$ a 2-uniform partition of $E(M)$ and $X \subseteq E(M)$ a set which is independent in $M$ and rainbow. Then there are three rainbow bases $B_1,B_2,B_3$ of $M$ such that $X\subseteq B_1 \cup B_2\cup B_3$.
\end{Lemma}

We are now ready to prove Theorem \ref{k41} using Lemmas \ref{fzu} and \ref{foresttree}.

\begin{proof}(of Theorem \ref{k41})
By assumption, there is a factorization of $M$ into $k$ bases $(B_1,\ldots,B_k)$. For all $i=1,\ldots,\alpha$, let $Z_i=B_{2i-1}\cup B_{2i}$. By Lemma \ref{fzu}, we obtain that $Z_i$ can be covered by 3 sets which are independent in $M|Z_i$ and hence in $M$ and rainbow, one of which is a basis of $M|Z_i$. Observe that, as $Z_i$ contains a basis of $M$,  this last set is also a basis of $M$. Further, for each $i = k-\beta+1,\ldots,k$, as $\mathcal{P}$ is 2-uniform, we clearly can cover $B_i$ by two sets which are rainbow. As these sets are contained in $B_i$, they are also independent in $M$. In total, we obtain that $E(M)$ can be covered by $3 \alpha+2\beta$ sets which are independent in $M$ and rainbow,  $\alpha$ of which are bases of $M$. By Lemma \ref{foresttree}, each of these sets which are not bases can be covered by 2 rainbow bases of $M$. We hence obtain that $E(M)$ can be covered by $5 \alpha+4 \beta$ bases of $M$ which are rainbow with respect to $\mathcal{P}$.
\end{proof}
Similarly, Lemmas \ref{fzu} and \ref{foresttree2} yield Theorem \ref{k412}.
\subsection{Covering double trees with rainbow spanning trees}\label{k2}

The objective of this section is to prove Theorem \ref{doublecover}. We first give some preliminary results we later apply to prove our result on 2-base matroids. After this, we give the main lemma from which Theorem \ref{doublecover} follows easily.

We first need the following general property of matroids.

\begin{Proposition}\label{contract}
Let $M$ be a matroid, $X \subseteq E(M)$ and $M'$ the direct sum of $M|X$ and $M/X$. Then for any $Z \subseteq E(M)$, we have $r_{M'}(Z)\leq r_M(Z)$.
\end{Proposition}
\begin{proof}
The contraction formula for the rank of a matroid and the submodularity of the rank function of $M$ yield \begin{align*}r_{M'}(Z)&=r_M(X \cap Z)+r_{M/X}(Z-X)\\&=r_M(X \cap Z)+r_M(Z \cup X)-r_M(X)\\&\leq r_M(Z).\end{align*}
\end{proof}
%We first give the following general result on matroids.
%We here prove the following result:
%\begin{Theorem}\label{doppelcover}
%Let $M$ be a 2-base matroid and $\mathcal{P}$ a 2-bounded subpartition of $E(M)$. Then $E(M)$ can be covered by $\log(|E(M)|)$ rainbow bases of $M$ which are rainbow with respect to $\mathcal{P}$.
%\end{Theorem}

%{\color{red} The proof also works for matroids.}

%We now collect some properties of 2-base matroids.
The following result can be found as Theorem 13.3.1 in \cite{book}.

\begin{Proposition}\label{sumf}
Let $M_1,\ldots,M_t$ be a collection of matroids on a common ground set, $M$ the sum of $M_1,\ldots,M_t$ and $Z \subseteq E(M)$. Then $$r_M(Z)=\min_{Y \subseteq Z}|Y|+\sum_{i=1}^tr_{M_i}(Z-Y).$$
\end{Proposition}
We are now ready to conclude the following result on matroids containing two bases.

\begin{Proposition}\label{gu1}
A matroid $M$ contains two disjoint bases if and only if $$|Z|+2r_M(E(M)-Z)\geq 2r_M(E(M))$$ holds for every $Z \subseteq E(M).$
\end{Proposition}
\begin{proof}
First suppose that $M$ contains two disjoint bases $B_1,B_2$ and let $Z \subseteq E(M).$ As $B_1$ and $B_2$ are bases of $M$, we obtain that $B_1-Z$ and $B_2-Z$ are independent in $M$.
 By the monotonicity of $r_M$, this yields

 \begin{align*}2r_M(E(M))&=|B_1\cup B_2|\\&=|B_1-Z|+|B_2-Z|+|(B_1\cup B_2)\cap Z|\\ &\leq r_M(B_1-Z)+r_M(B_2-Z)+|Z|\\&\leq |Z|+2r_M(E(M)-Z).\end{align*} 

For the other direction, let $M'$ be the matroid that is the sum of two copies $M_1,M_2$ of $M$ and let $Y \subseteq E(M)$. We have $$|Y|+\sum_{i=1}^2r_{M_i}(E(M)-Y)\geq 2r_M(E(M))$$ by assumption. By Proposition \ref{sumf}, we obtain that $r_{M'}(E(M))=2r_M(E(M))$. By definition of $M'$, we obtain that $M$ contains two disjoint bases.
\end{proof}
We are now ready to give the following characterization of 2-base matroids.
\begin{Proposition}\label{gu2}
A matroid $M$ is a 2-base matroid if and only if $r_M(E(M))=\frac{1}{2}|E(M)|$ and $r_M(Z)\geq \frac{1}{2}|Z|$ holds for every $Z \subseteq E(M)$.
\end{Proposition}
\begin{proof}
First suppose that $M$ is a 2-base matroid, so $E(M)$ can be partitioned into two bases $B_1,B_2$ of $M$. As $|B_1|=|B_2|$, we obtain $r_M(E(M))=|B_1|=\frac{1}{2}|E(M)|$. Further, for every $Z \subseteq E(M)$, we have $$r_M(Z)\geq \max\{|B_1\cap Z|, |B_2 \cap Z|\}\geq \frac{1}{2}|Z|.$$

Now suppose that $r_M(E(M))=\frac{1}{2}|E(M)|$ and $r_M(Z)\geq \frac{1}{2}|Z|$ holds for every $Z \subseteq E(M)$. It suffices to prove that $M$ contains two disjoint bases. Consider some $Z \subseteq E(M)$. By assumption, we have $$|Z|+2r_M(E(M)-Z) \geq |Z|+|E(M)-Z|=|E(M)|=2r_M(E(M)).$$ Hence the statement follows from Proposition \ref{gu1}.
\end{proof}
The next result allows some modifications on 2-base matroids and will be applied in the proof of the main lemma.

For a 2-base matroid $M$, a set $X \subseteq E(M)$ is called {\it tight} if $M|X$ is also a 2-base matroid. We say that $X$ is {\it trivial} if  $X=\emptyset$ or $X= E(M)$, {\it nontrivial} otherwise.

\begin{Proposition}\label{kontrahierbaum}
Let $M$ be a 2-base matroid and $X \subseteq E(M)$ a tight set. Then $M/X$ is also a 2-base matroid.
\end{Proposition}
\begin{proof}
First observe that as both $M$ and $M|X$ are 2-base matroids and by Proposition \ref{gu2}, we have $$r_{M/X}(E(M/X))=r_M(E(M))-r_M(X)=\frac{1}{2}|E(M)|-\frac{1}{2}|X|=\frac{1}{2}|E(M/X)|.$$ 

Further, for any $Z \subseteq E(M)-X$, as $X$ is a tight set in $M$, we have $$r_{M/X}(Z)=r_M(X \cup Z)-r_M(X)\geq \frac{1}{2}|X\cup Z|-\frac{1}{2}|X|=\frac{1}{2}| Z|.$$ The statement hence follows from Proposition \ref{gu2}.%Now let $B_1,B_2$ be two disjoint bases of $M$. As $M$ is a 2-base matroid, we have $B_1\cup B_2=E(M)$, so $(B_1 \cap X) \cup (B_2 \cap X)=X$. As $B_i \cap X$ is independent in $M|X$ and $M|X$ is a 2-base matroid, we have $|B_i \cap X|\leq r_{M|X}(X)=\frac{1}{2}|X|$ for $i=1,2$. This yields  $|X|=\frac{1}{2}|X|+\frac{1}{2}|X|\geq |B_1 \cap X|+|B_2 \cap X|=|X|$, so $|B_1 \cap X|=|B_2 \cap X|=\frac{1}{2}|X|$. For $i=1,2$, we hence obtain $r_{M/X}(B_i-X)=r_M(B_i \cup X)-r_M(X)=r_M(E(M))-r_M(X)=\frac{1}{2}|E(M)|-\frac{1}{2}|X|=r_{M/X}(E(M/X))$ and $r_{M/X}(E(M/X))=\frac{1}{2}|E(M/X)|=\frac{1}{2}(|E(M)|-|X|)=|B_i|-|B_i \cap X|=|B_i -X|$. Hence $(B_1-X,B_2-X)$ is a partition of $E((M/X))$ into two bases of $M/X$. 
\end{proof}

The following result allows to benefit from the absence of a tight set. %It would be interesting to see if it has further applications.
\begin{Proposition}\label{delcon}
Let $M$ be a 2-base matroid that does not contain a nontrivial tight set and $e,f \in E(M)$. Then $M/e-f$ is a 2-base matroid. 
\end{Proposition}
\begin{proof}
Let $M_0=M/e-f$. 

We need the following two claims.
\begin{Claim}\label{un}
Let $\emptyset \neq X \subsetneq E(M)$. Then $r_{M}(X)> \frac{1}{2}|X|.$
\end{Claim}
\begin{proof}
As $M$ is a 2-base matroid, $E(M)$ can be partitioned into two disjoint bases $B_1,B_2$ of $M$. We obtain that $(B_1\cap X, B_2 \cap X)$ is a partition of $X$ into sets which are independent in $M$. This yields $$r_M(X)\geq \max\{|B_1 \cap X|,|B_2 \cap X|\}\geq \frac{1}{2}|X|.$$ Further, if equality holds throughout, then $(B_1\cap X, B_2 \cap X)$ is a partition of $X$ into bases of $M|X$, so $X$ is a tight set contradicting the assumption. We hence obtain $r_M(X)>\frac{1}{2}|X|$.
\end{proof}

\begin{Claim}\label{deux}
Let $Y \subseteq E(M_0)$. Then $r_{M_0}(Y)\geq \frac{1}{2}|Y|.$
\end{Claim}
\begin{proof}

We obtain by Claim \ref{un} that $$r_{M_0}(Y)=r_{M/e}(Y)=r_M(Y \cup e)-r_M(e)\geq \frac{1}{2}(|Y \cup e|+1)-1= \frac{1}{2}|Y|.$$
\end{proof}

Further observe that as $M$ is a 2-base matroid, we have \begin{align*}r_{M_0}(E(M_0))&=r_M(E(M)-f)-r_M(e)\\&=r_M(E(M))-1\\&=\frac{1}{2}|E(M)|-1\\&=\frac{1}{2}|E(M_0)|.\end{align*}
The statement now follows from Proposition \ref{gu2}.% and $r_{M_0^*}(E(M_0))=|E(M_0)|-r_{M_0}(E(M_0))=r_{M_0}(E(M_0))$.
\end{proof}

We are now ready to prove the main lemma from which Theorem \ref{doublecover} follows easily.

\begin{Lemma}\label{main}
Let $M$ be a matroid that is the direct sum of 2-base matroids $M_1,\ldots,M_t$, $Z \subseteq E(M)$ and $\mathcal{P}$ a 2-uniform partition of $E(M)$. Then there is a basis $B$ of $M$ such that $B$ is rainbow and $|Z \cap B|\geq \frac{|Z|}{2}$.
\end{Lemma}
\begin{proof}
Observe that a set $B \subseteq E(M)$ is a basis of $M$ if and only if $B \cap E(M_i)$ is a basis of $M_i$ for $i=1,\ldots,t$.
Suppose for the sake of a contradiction that $(M,M_1,\ldots,M_t)$ is a counterexample that is minimum with respect to the number of elements of $M$  and subject to this, it is minimum with respect to the total number of nontrivial tight sets contained in $M_1,\ldots,M_t$. %We may suppose that $\mathcal{P}$ is a partition of $E(M)$ with $|X|=2$ for all $X \in \mathcal{P}$.

 First suppose that there is some $i\in \{1,\ldots,t\}$, say $t$, such that $M_i$ contains a nontrivial tight set $X$. Let $M'$ be the matroid that is the direct sum of $M'_1,\ldots,M'_{t+1}$ where $M'_i=M_i$ for all $i=1,\ldots,t-1$, $M'_t=M_t|X$ and $M'_{t+1}=M_t/X$. Observe that by Proposition \ref{kontrahierbaum}, $M'_i$ is a 2-base matroid for $i=1,\ldots,t+1$. Further, every nontrivial tight set in one of $M_1'\ldots,M'_{t+1}$ also is a nontrivial tight set in one of $M_1,\ldots,M_t$ and $X$ is a nontrivial tight set in $M_t$ but none of $M_1'\ldots,M'_{t+1}$. We obtain that the total number of nontrivial tight sets in $M'_1,\ldots,M'_{t+1}$ is less than in $M_1,\ldots,M_t$. Hence, by minimality, there is a basis $B$  of $M'$ such that $B \cap E(M'_i)$  is rainbow and $|Z \cap B|\geq \frac{|Z|}{2}$. By Proposition \ref{contract}, we have $$r_M(B)\geq r_{M'}(B)= \frac{|E(M')|}{2}=\frac{|E(M)|}{2},$$ so $B$ is a basis of $M$. This contradicts $M$ being a counterexample.

Now suppose that none of $M_1,\ldots,M_t$ contains a nontrivial tight set.
\begin{Claim}
There is a  sequence $e_1,f_1,\ldots,e_\mu,f_\mu$ of distinct elements of  $E(M)$ with the following properties:
\begin{enumerate}[(i)]
\item $\{e_j,f_j\} \in \mathcal{P}$ for all $j=1,\ldots,\mu$,
\item for every $j=2,\ldots,\mu$, there is some  $i$ such that $\{f_{j-1},e_j\}\subseteq E(M_i)$,
\item $|E(M_i)\cap \{e_1,\ldots,e_\mu\}|,|E(M_i)\cap \{f_1,\ldots,f_\mu\}|\leq 1$ for all $i=1,\ldots,t$,
\item there is some  $i$ such that $\{f_{\mu},e_1\}\subseteq E(M_i)$,
\end{enumerate}
\end{Claim}
\begin{proof}
Consider a longest sequence satisfying $(i)-(iii)$. Let $\ell$ be the unique integer such that $f_{\mu}$ is contained in $E(M_{\ell})$. If there is some $j \in \{1,\ldots,\mu\}$ such that $e_j$ is contained in $E(M_{\ell})$, then $e_j,f_j,\ldots,e_\mu,f_\mu$ is a sequence satisfying $(i)-(iv)$, so we are done. Otherwise, as $M_\ell$ is a 2-base matroid and $|E(M_\ell)\cap \{f_1,\ldots,f_\mu\}|\leq 1$, we can pick some element $e_{\mu+1}\in E(M_\ell)-f_\mu$. Let $f_{\mu+1}$ be the unique element of $E(M)$ with $\{e_{\mu+1},f_{\mu+1}\} \in \mathcal{P}$. By construction, we have $f_{\mu+1}\in E(M)-\{e_1,\ldots,e_{\mu+1},f_1,\ldots,f_{\mu}\}$. Let $q$ be the unique integer such that $f_{\mu+1}\in E(M_q)$. If $|E(M_q)\cap \{f_1,\ldots,f_\mu\}|=0$, then $e_1,f_1,\ldots,e_{\mu+1},f_{\mu+1}$ is a longer sequence satisfying $(i)-(iii)$, a contradiction. Otherwise there is some $j\leq \mu$ such that $f_j\in E(M_q)$ and so $e_{j+1},f_{j+1},\ldots,e_{\mu+1},f_{\mu+1}$ is a sequence satisfying $(i)-(iv)$. This finishes the proof.

%If $S$ is cyclic, there is nothing to prove. First suppose that $S=e_1,f_1,\ldots,e_r$ for some positive integer $r$. Let $a \in E(G)$ be the unique edge such that $\{e_r,a\} \in \mathcal{P}$ and let $i^*$ be the unique integer such that $a \in G_{i^*}$. If $|E(G_{i^*})\cap \{f_1,\ldots,f_\mu\}|= 0$, then $e_1,f_1,\ldots,e_r,a$ is a longer sequence than $S$ satisfying $(i)-(iii)$, a contradiction to the choice of $S$. Otherwise, there is some $1 \leq r' \leq r$ such that $f_{r'-1},e_{r'}\in E(G_{i^*}) $. Now $e_{r'},f_{r'},\ldots,e_r,a$ is a cyclic sequence of the desired form.

%Now suppose that $S=e_1,f_1,\ldots,e_r,f_r$ for some positive integer $r$ and let $i^*$ be the unique integer such that $f_r \in G_{i^*}$. If $|E(G_{i^*})\cap \{e_1,\ldots,e_\mu\}|=1$, we obtain by assumption that $e_1 \in E(G_{i^*})$ and hence $S$ is cyclic, so there is nothing to prove. Otherwise, we can choose some arbitrary $a \in E(G_{i^*})-f_r$ which exists as $|E(G_{i^*})|\geq 2$. Now $S=e_1,f_1,\ldots,e_r,f_r,a$ is a sequence which is longer than $S$ and satisfies $(i)-(iii)$, a contradiction. 
\end{proof}

For this paragraph, we set $f_0=f_{\mu}$ and the indices are considered $\mod \mu$. By symmetry, we may suppose that $|\{e_1,\ldots,e_\mu\}\cap Z| \geq |\{f_1,\ldots,f_\mu\}\cap Z|$. Now for every $i \in \{1,\ldots,t\}$ such that $\{f_{j-1},e_j\}\subseteq E(M_i)$ for some $j \in \{1,\ldots,\mu\}$, let $M'_i=M_i/e_j-f_{j-1 }$ and for all $i \in \{1,\ldots,t\}$ such that $\{f_{j-1},e_j\}\cap E(M_i)=\emptyset$ for all $j \in \{1,\ldots,\mu\}$, let $M'_i=M_i$. By Proposition \ref{delcon}, we obtain that $M'_i$ is a 2-base matroid for all $i =1,\ldots,t$. Let $M'$ be the direct sum of $M'_i,i=1,\ldots,t$. As $M'$ is smaller than $M$, there is a basis $B'$ of $M'$ which is rainbow such that $|Z \cap B'|\geq \frac{|Z\cap E(M')|}{2}$. Now let $B=B'\cup \{e_1,\ldots,e_\mu\}$. Observe that as the $e_j$ are non-loop elements in different components of $M$, we have $$r_M(B)=r_{M'}(B')+r_M(\{e_1,\ldots,e_\mu\})=|B'|+\mu=|B|.$$ As $$|B|=\mu+|B'|=\frac{1}{2}(2\mu+|E(M')|)=\frac{1}{2}|E(M)|,$$ we obtain that $B$ is a basis of $M$. Further, by construction, $B$ is rainbow with respect to $\mathcal{P}$. Finally, we have \begin{align*}|Z \cap B|&=|Z \cap B'|+|Z \cap \{e_1,\ldots,e_\mu\}|\\ &\geq \frac{|Z \cap E(M')|}{2}+\frac{|Z \cap \{e_1,\ldots,e_\mu,f_1,\ldots,f_\mu\}|}{2}\\&=\frac{|Z|}{2},\end{align*} a contradiction to $M$ being a counterexample.
\end{proof}

We are now ready to prove Theorem \ref{doublecover}. For technical reasons, we prove the following slightly stronger statement.

\begin{Theorem}
Let $M$ be a 2-base matroid and $\mathcal{P}$ a 2-bounded subpartition of $E(M)$. Then there is a set of $\log(|Z|)+1$ rainbow bases of $M$ that covers $Z$.
\end{Theorem}
\begin{proof}
Clearly, we may suppose that $\mathcal{P}$ is 2-uniform. We proceed by induction on $|Z|$. For $|Z|=1$, the statement immediately follows from Lemma \ref{main}. We may hence suppose that $|Z|\geq 2$. By Lemma \ref{main}, there is a basis $B$ of $M$ which is rainbow with respect to $\mathcal{P}$ and satisfies $|B \cap Z|\geq \frac{1}{2}|Z|$. Inductively, $Z-B$ can be covered by a set of $\log(|Z-B|)+1$ rainbow bases of $M$. Hence $Z$ can be covered by a set of $\log(|Z-B|)+2$ bases of $M$ which are rainbow with respect to $\mathcal{P}$. As $\log(|Z-B|)\leq \log(\frac{1}{2}|Z|)=\log(|Z|)-1$, the statement follows.
\end{proof}
\section{Conclusion}\label{conc}
We deal with packing and covering problems involving the common bases of two matroids. Our work leaves many questions open, for example Conjectures \ref{optim} and \ref{optim2}. It is not even clear whether the following much weaker statement is true. We think a proof of it could be a first step in the direction of Conjectures \ref{optim} and \ref{optim2}.

\begin{Conjecture}\label{optim3}
There is an integer $k$ such that for every $k$-base matroid and every 2-bounded partition $\mathcal{P}$ of $E(M)$, we can cover $M$ by $k$ bases one of which is rainbow.
\end{Conjecture}

Further, we would like to understand whether the results in Section \ref{cover} could be extended to the case $k=2$. If true, the following would be an improvement on Theorem \ref{doublecover}.
\begin{Conjecture}
There is an integer $k$ such that for every 2-base matroid $M$ and every 2-bounded partition of $E(M)$, we can cover $E(M)$ by a set of $k$ rainbow bases.
\end{Conjecture}
\section*{Acknowledgement}
We wish to thank K. Bérczi and T. Kir\'aly who made us aware of Conjecture \ref{ahconj}.


\begin{thebibliography}{99}
%\bibitem{ab} R. Aharoni, E. Berger, The intersection of a matroid and a simplicial complex, Trans. Am. Math. Soc. 358, 4895-4917, 2006,
%\bibitem{abs} R. Aharoni, E. Berger, P. Sprüssel, Two disjoint independent sets in matroid-graph pairs, Graphs Comb. 31, 1107-1116, 2015,
\bibitem{AB} R. Aharoni, E. Berger, The intersection of a matroid and a simplicial complex, Transactions of the  American Mathematical Society, 358, 4895--4917, 2006.
\bibitem{ABZ} R. Aharoni, E. Berger, R. Ziv, The edge covering number of the intersection of two matroids, Discrete Mathematics, 312, 81--85, 2012.
\bibitem{AH} R. Aharoni, E. Hallufgil, Coloring by two-way independent sets, Discrete Mathematics, 309, 4853--4860, 2009.
\bibitem{bck} K. Bérczi, G. Cs\'aji, T. Kir\'aly, On the complexity of packing rainbow spanning trees, Discrete Mathematics, 346(4), 113297, 2023.
\bibitem{bs}  K. Bérczi, T. Schwarcz, Complexity of packing common bases in matroids, Mathematical Programming, 188, 1-18, 2021.
\bibitem{bh} R. Brualdi, S. Hollingsworth, Multicolored trees in complete graphs, Journal of Combinatorial Theory, Series B, 68, 310–313, 1996.
\bibitem{chow} T. Y. Chow, Reduction of Rota's basis conjecture to a problem on three bases, SIAM Journal on Discrete Mathematics, 23(1), 369-371, 2009.
\bibitem{e4} J. Edmonds, Some well-solved problems in combinatorial optimization, in: Combinatorial Programming: Methods and Applications (Proceedings NATO Advanced Study Institute, Versailles, 1974. B. Roy, ed.), Reidel, Dordrecht, pp. 285-301, 1975.
\bibitem{e3} J. Edmonds, Submodular functions, matroids, and certain polyhedra, in: eds. R. Guy, H. Hanani, N. Sauer, and J. Sch{\"o}nheim, Combinatorial Structures and their Applications, Gordon and Breach, New York, pp. 69-70, 1970.
\bibitem{fs} Z. Fekete, J. Szab\'o, Uniform partitioning to bases in a matroid, Technical Report TR-2005-03, Egerv\' ary Research Group, Budapest, 2003, \url{https://web.cs.elte.hu/egres/tr/egres-05-03.pdf}.
\bibitem{fsj} Z. Fekete, J. Szab\'o, Equitable partitions to spanning trees in a graph, Electronic Journal of Combinatorics, 18(1), 2011.
\bibitem{book} A. Frank,  Connections in Combinatorial Optimization, Oxford University Press, 2011.
\bibitem{f1} Egres Open Problems, Problem: `{\it In-degree bounded directed forests}', \url{http://lemon.cs.elte.hu/egres/open/In-degree_bounded_directed_forests}.
\bibitem{gj} M. R. Garey, D. S. Johnson, Computers and Intractability: A Guide to the Theory of NP-Completeness, W.H. Freeman, 1979.
\bibitem{gkmo} S. Glock, D. K\"uhn, R. Montgomery, D. Osthus, Decompositions into isomorphic rainbow spanning trees, Journal of Combinatorial Theory, Series B, 146, 439-484, 2021.
\bibitem{hkl} N. J. A. Harvey, T. Király, L. C. Lau, On disjoint common bases in two matroids, SIAM Journal on Discrete Mathematics, 25(4), 1792-1803, 2011.
\bibitem{h} P. Horn, Rainbow spanning trees in complete graphs colored by one-factorizations, Journal of Graph Theory, 87, 333-346, 2018.
\bibitem{hn} P. Horn, L. Nelsen, Many edge-disjoint rainbow spanning trees in general graphs, preprint, \url{https://arxiv.org/abs/1704.00048}.
\bibitem{t} Egres Open Problems, Problem: `{\it In-degree bounded directed forests}', Discussion page, \url{http://lemon.cs.elte.hu/egres/open/Talk:In-degree_bounded_directed_forests}.
\bibitem{ps} A. Pokrovskiy, B. Sudakov, Linearly many rainbow trees in properly edge-coloured complete graphs,Journal of Combinatorial Theory, Series B, 132, 134-156, 2018.
\bibitem{schaefer} T. J. Schaefer, The Complexity of Satisfiability Problems,   Proceedings of the Tenth Annual ACM Symposium on Theory of Computing, STOC ’78, 3, 216-226, 1978.
\end{thebibliography}
\end{document}